\documentclass[a4, 12pt]{amsart}
\usepackage[mathscr]{eucal}
\usepackage{amssymb}
\usepackage{latexsym}
\usepackage{amsthm}
\theoremstyle{plain}
\newtheorem{theorem}{Theorem}[section]

\newtheorem{corollary}{Corollary}[section]
\newtheorem{remark}{Remark}[section]
\newtheorem{lemma}{Lemma}[section]

\makeatletter
\@addtoreset{equation}{section}

\begin{document}

\title[area of minimal hypersurface ]
{ Area of minimal hypersurfaces }
\author{Qing-Ming Cheng, Guoxin Wei and Yuting Zeng}
\address{Qing-Ming Cheng \\ Department of Applied Mathematics, Faculty of Sciences ,
Fukuoka  University, 814-0180, Fukuoka,  Japan, cheng@fukuoka-u.ac.jp}
\address{Guoxin Wei \\  School of Mathematical Sciences, South China Normal University,
510631, Guangzhou,  China, weiguoxin@tsinghua.org.cn}
\address{Yuting Zeng \\  School of Mathematical Sciences, South China Normal University,
510631, Guangzhou,  China, 1054237466@qq.com}

\begin{abstract}
A well-known conjecture of Yau states that the area of one of Clifford minimal hypersurfaces $S^k\big{(}\sqrt{\frac{k}{n}}\, \big{)}\times S^{n-k}\big{(}\sqrt{\frac{n-k}{n}}\, \big{)}$ gives the lowest value of area among all non-totally geodesic compact minimal hypersurfaces in the unit sphere $S^{n+1}(1)$. The present paper shows that Yau conjecture is true for minimal rotational hypersurfaces, more precisely, the area $|M^n|$ of compact minimal rotational hypersurface $M^n$ is either equal to $|S^n(1)|$, or equal to $|S^1(\sqrt{\frac{1}{n}})\times S^{n-1}(\sqrt{\frac{n-1}{n}})|$, or greater than $2(1-\frac{1}{\pi})|S^1(\sqrt{\frac{1}{n}})\times S^{n-1}(\sqrt{\frac{n-1}{n}})|$. As the application, the entropies of some special self-shrinkers are estimated.
\end{abstract}




\footnotetext{The first author was partially  supported by JSPS Grant-in-Aid for Scientific Research (B):  No.16H03937
and Challenging Exploratory Research.
The second author was partly supported by grant No. 11371150 of NSFC.}

\maketitle
\section{Introduction}

The study of minimal hypersurfaces in space forms (that is, $\mathbb{R}^{n+1}$, the sphere $S^{n+1}$, and hyperbolic space $H^{n+1}$), is one of the most important subjects in differential geometry. There are a lot of nice results on this topic (see \cite{B2}, \cite{CDK}, \cite{DX}, \cite{L}, \cite{PT} and many others). The simplest examples of minimal hypersurfaces in $S^{n+1}$ are the totally geodesic $n$-spheres. Another basic examples are the so-called Clifford minimal hypersurfaces $S^k\big{(}\sqrt{\frac{k}{n}}\, \big{)}\times S^{n-k}\big{(}\sqrt{\frac{n-k}{n}}\, \big{)}$.

Cheng, Li and Yau \cite{CLY} proved in 1984 that if $M^n$ is a compact minimal hypersurface in the unit sphere $S^{n+1}(1)$ and $M^n$ is not totally geodesic, then there exists a constant $c(n)>0$, such that the area $|M^n|$ of $M^n$ satisfies $|M^n|>(1+c(n))|S^n(1)|$, that is, the area of the totally geodesic $n$-sphere $S^{n}(1)\subset S^{n+1}(1)$ is the smallest among all compact minimal hypersurfaces in $S^{n+1}(1)$.

In 1992, S.T. Yau \cite{Y}  posed the following conjecture (P288, Problem 31):

\noindent {\bf Yau Conjecture}: {\it The area of one of Clifford minimal hypersurfaces $S^k\big{(}\sqrt{\frac{k}{n}}\, \big{)}\times S^{n-k}\big{(}\sqrt{\frac{n-k}{n}}\, \big{)}$ gives the lowest value of area among all non-totally geodesic compact minimal hypersurfaces in the unit sphere $S^{n+1}(1)$.}

In this paper, we consider a little more restricted problem of Yau conjecture for compact minimal rotational hypersurfaces $M^n$ in $S^{n+1}(1)$. As one of the main results of this paper, we prove

\begin{theorem}\label{theorem 1}
If $M^n$ is a compact minimal rotational hypersurface in $S^{n+1}(1)$, then the area $|M^n|$ of $M^n$ satisfies either
$|M^n|=|S^n(1)|$, or $|M^n|=|S^{1}(\sqrt{\frac{1}{n}})\times S^{n-1}(\sqrt{\frac{n-1}{n}})|$,  or
$|M^n|> 2(1-\dfrac{1}{\pi})|S^{1}(\sqrt{\frac{1}{n}})\times S^{n-1}(\sqrt{\frac{n-1}{n}})|$.
\end{theorem}

\begin{remark}\label{remark 1}
From the theorem \ref{theorem 1}, we can know that Yau conjecture is true for minimal rotational hypersurfaces.
\end{remark}

\begin{corollary}\label{theorem 2}
If $M^n$ is a compact minimal rotational hypersurface in $S^{n+1}(1)$, then the area $|M^n|$ of $M^n$ satisfies either
$|M^n|=|S^n(1)|$, or $|M^n|=|S^{1}(\sqrt{\frac{1}{n}})\times S^{n-1}(\sqrt{\frac{n-1}{n}})|$,  or
$|M^n|=|M^n(3,2)|$, or $|M^n|> 3(1-\dfrac{1}{\pi})|S^{1}(\sqrt{\frac{1}{n}})\times S^{n-1}(\sqrt{\frac{n-1}{n}})|$, where $M^n(3,2)$ is the compact minimal rotational hypersurface with $3$-fold rotational symmetry and rotation number $2$.
\end{corollary}

\begin{remark}\label{remark 2}
From the theorem \ref{theorem 2}, we have that the conjecture proposed by Perdomo and the first author in \cite{PW} is true except the case $M^n(3,2)$ since
$3(1-\dfrac{1}{\pi})>2$.
\end{remark}

By considering the upper bounds of some integral, we show another main result of the present paper concerning the areas of compact minimal rotational hypersurfaces $M^n$ in $S^{n+1}(1)$, stated as follows

\begin{theorem}\label{theorem 3}
The lowest value of area among all compact minimal rotational hypersurfaces with non-constant principal curvatures in the unit sphere $S^{n+1}(1)$ is the area of either $M^n(3,2)$ or $M^n(5,3)$, where $M^n(k,l)$ is a compact minimal rotational hypersurface in $S^{n+1}(1)$ with $k$-fold rotational symmetry and rotation number $l$, its periodic is $K=l\times\frac{2\pi}{k}=\frac{2l\pi}{k}$.
\end{theorem}

%

\section{Preliminaries}

Without loss of generality, we assume that the minimal rotational hypersurface is neither $S^n(1)$ nor $S^{1}(\sqrt{\frac{1}{n}})\times S^{n-1}(\sqrt{\frac{n-1}{n}})$.
From \cite{O} and \cite{P1}, we have the following description for every minimal rotational hypersurface in $S^{n+1}(1)$ whose principal curvatures are not constant.

Let us describe complete minimal rotational hypersurface $M_a$. For any positive number $a<a_0=\frac{(n-1)^{n-1}}{n^n}$, let $r(t)$ be a solution of the following ordinary differential equation

\begin{eqnarray}\label{theode}
(r^\prime(t))^2=1-r(t)^2-a r(t)^{2-2n}.
\end{eqnarray}
Since $0<a<a_0$, we have that the function $q(v)=1-v^2-a v^{2-2n}$ has two positive roots $r_1$ and $r_2$ between $0$ and $1$. Therefore it is not difficult to check that the solution of the differential equation (\ref{theode}) is a periodic function with period $T=2 \int_{r_1}^{r_2} \frac{1}{\sqrt{q(v)}}\, dv$  that takes values between $r_1$ and $r_2$. Moreover, since the differential equation (\ref{theode}) does not depend on $t$ explicitly, then, for any $k$ we have that $r(t-k)$ is a solution, provided $r(t)$ is a solution. Therefore we can assume that $r(0)=r_1$ and $r(\frac{T}{2})=r_2$.

If we define

$$\theta(t)= \int_0^t \frac{\sqrt{a} \, r^{1-n}(\tau)}{1-r^2(\tau)}\, d\tau,$$
then the hypersurface $\phi:S^{n-1}\times [0, T]\to S^{n+1}$ given by

$$\phi(y,t)=\big{(}\, r(t)\, y, \, \sqrt{1-r^2(t)}\,  \cos(\theta(t)),\,  \sqrt{1-r^2(t)}\,  \sin(\theta(t))\, \big{)}$$
is called the {\it fundamental portion} of $M_a$. The curve
$$\alpha(t)=\big{(}\, \sqrt{1-r^2(t)}\,  \cos(\theta(t)),\,  \sqrt{1-r^2(t)}\,  \sin(\theta(t))\, \big{)}$$
 is called the {\it profile curve} of $M_a$.  It turns out that the whole hypersurface $M_a$ is the union of rotations of the fundamental portion. We also have that the hypersurface $M_a$ is compact if and only if the number (also see \cite{O})

\begin{eqnarray}\label{k}
K(a)=\theta(T)= 2 \int_0^{\frac{T}{2}} \frac{\sqrt{a} \, r^{1-n}(\tau)}{1-r^2(\tau)}\, d\tau =2 \pi \frac{p}{s},
\end{eqnarray}
for some pair of relatively prime integers $p$ and $s$. In this case, $M_a$ is made out of exactly $s$ copies of the fundamental portion, that is, $M^n$ is a hypersurface with {\it $s$-fold rotational symmetry} and {\it rotation number $p$}. When $\frac{K(a)}{2 \pi}$ is not a rational number, we have that the hypersurface is not compact.

The following lemma is due to  Otsuki \cite{O1}, \cite{O2}. A proof for the particular case $n=2$ can also be found in  \cite{AL} and in \cite{P2}.

\begin{lemma}\label{lemma 1}
The function $K(a)$ given in (\ref{k}) is strictly increasing and differentiable on $(0,a_0)$ and

$$ \lim_{a\to 0} K(a)=\pi ,\qquad   \lim_{a\to a_0} K(a)=\sqrt{2}\pi.  $$

\end{lemma}

In \cite{PW}, Perdomo and the first author proved the  following lemma,

\begin{lemma}\label{lemma 2}
If $M^n$ is a compact minimal rotational hypersurface in $S^{n+1}(1)$ with non-constant principal curvatures, then the area of $M^n$, denoted by $|M^n|$, is equal to $$w(a)p,$$
 where $p$ is the rotational number greater than $1$ (see \eqref{k}),
\begin{equation}
w(a)=2\pi \sigma_{n-1}\dfrac{\int_{x_1}^{x_2}\frac{x^{n-\frac{3}{2}} \,}{\sqrt{x^{n-1}-x^n-a}} dx}{K(a)}\,,
\end{equation}
$a\in (0,a_0)$, $x_1<x_2$ are the only two roots in the interval $(0,1)$ of the polynomial $z(x)=x^{n-1}-x^n-a$,
$\sigma_{n-1}$ denotes the area of $S^{n-1}(1)$.
Moreover, we have
\begin{equation}
\aligned
&\ \ \ \lim_{a\to a_0} w(a)\\
&=\lim_{a\to a_0}2\pi \sigma_{n-1}\dfrac{\int_{x_1}^{x_2}\frac{x^{n-\frac{3}{2}} \,}{\sqrt{x^{n-1}-x^n-a}} dx}{K(a)}\\
 &=2\pi \sigma_{n-1}\dfrac{\sqrt{2a_0}\pi}{\sqrt{2}\pi}=2\pi \sigma_{n-1}\sqrt{a_0}\\
&= \biggl|S^1\biggl(\sqrt{\frac{1}{n}} \biggl)\times  S^{n-1} \biggl(\sqrt{\frac{n-1}{n}} \biggl) \biggl|.
\endaligned
\end{equation}

\end{lemma}

\section{Proofs of the theorem \ref{theorem 1} and the corollary \ref{theorem 2}}

\noindent{\it Proof of the theorem \ref{theorem 1}}.
From the lemma \ref{lemma 1} and lemma \ref{lemma 2}, it is sufficient to prove the theorem \ref{theorem 1} if we can get the following inequality

\begin{equation}\label{eq:3-21-1}
2\int_{x_1}^{x_2}\frac{x^{n-\frac{3}{2}}}{\sqrt{x^{n-1}-x^{n}-a}}dx>2(1-\frac{1}{\pi})\sqrt{2a_0}\pi,
\end{equation}
for a compact minimal rotational hypersurface $M^n$ in $S^{n+1}(1)$ with non-constant principal curvatures,
where $0<x_1<x_0=\frac{n-1}{n}<x_2<1$ are two roots of $z(x)=x^{n-1}-x^{n}-a$,
$0<a<a_0$.

%
%
%

 We next prove \eqref{eq:3-21-1}. Let
\begin{equation}
y=x^{n-\frac{1}{2}},
\end{equation}
then
\begin{equation}
2\int_{x_1}^{x_2}\frac{x^{n-\frac{3}{2}}}{\sqrt{x^{n-1}-x^{n}-a}}dx
=\frac{4}{2n-1}\int_{y_1}^{y_2}\frac{1}{\sqrt{y^{\frac{2n-2}{2n-1}}-y^{\frac{2n}{2n-1}}-a}}dy.
\end{equation}
\eqref{eq:3-21-1} is equivalent to
\begin{equation}
\int_{y_1}^{y_2}\frac{1}{\sqrt{y^{\frac{2n-2}{2n-1}}-y^{\frac{2n}{2n-1}}-a}}dy>2(1-\frac{1}{\pi})\frac{2n-1}{4}\sqrt{2a_0}\pi:=2(1-\frac{1}{\pi})A_0\pi,
\end{equation}
where $y_1$, $y_2$ are two roots of $f(y)=y^{\frac{2n-2}{2n-1}}-y^{\frac{2n}{2n-1}}-a$, $A_0=\frac{2n-1}{4}\sqrt{2a_0}$.

 We construct a function $g_1(y)$ as follows:
\begin{equation}
 g_1(y) = {\begin{cases}
     c{{(\sqrt {{y_1}}  - \sqrt {{y_c}} )}^2} - c{{(\sqrt y  - \sqrt {{y_c}} )}^2},\ \ y \in [{y_1},{y_c}],\\[3mm]
     b{{({y_2} - {y_c})}^2} - b{{(y - {y_c})}^2},\quad \quad \quad \ \ \ \ y \in ({y_c},{y_2}],
     \end{cases}}
  \end{equation}
where $c = \frac{{8\sqrt {n(n - 1)} }}{{{{(2n - 1)}^2}}}$, $b = \frac{{2(n - 1)}}{{{{(2n - 1)}^2}}}{\left( {\frac{{n - 1}}{n}} \right)^{ - n}}$ and
$y_c=(\frac{n-1}{n})^{\frac{1}{2}(2n-1)}$.

Let
\begin{equation}
h_1(y)= g_1(y) - f(y), \ \ \ {\mbox{for}} \ y\in [y_1,y_2],
\end{equation}
 we will prove
\begin{equation}
h_1(y)\geq 0
\end{equation}
and
\begin{equation}
\aligned
 &\ \ \ \int_{y_1}^{y_2}\frac{1}{\sqrt{y^{\frac{2n-2}{2n-1}}-y^{\frac{2n}{2n-1}}-a}}dy\\
 &=\int_{{y_1}}^{{y_2}} {\frac{1}{{\sqrt {f(y)} }}} dy \ge \int_{{y_1}}^{{y_2}} {\frac{1}{{\sqrt {g_1(y)} }}} dy\\
& > \left( {2-\frac{2}{\pi} + \frac{2}{\pi }\sqrt {\frac{{{y_1}}}{{{y_c}}}} } \right){A_0}\pi.
\endaligned
\end{equation}

\noindent We next consider two cases.\\

\noindent {\bf Case 1: $y\in [y_c,y_2]$.}

 By a direct calculation, we obtain
\begin{equation}
h_1^{\prime}(y) =  - 2b(y - {y_c}) - \frac{{2n - 2}}{{2n - 1}}{y^{\frac{{ - 1}}{{2n - 1}}}} + \frac{{2n}}{{2n - 1}}{y^{\frac{1}{{2n - 1}}}}, \end{equation}
\begin{equation}\label{eq:4-19-1}
h_1^{\prime\prime}(y) =  - 2b + \frac{{2n - 2}}{{2n - 1}}\frac{1}{{2n - 1}}{y^{ - \frac{{2n}}{{2n - 1}}}} + \frac{{2n}}{{2n - 1}}\frac{1}{{2n - 1}}{y^{ - \frac{{2n - 2}}{{2n - 1}}}},
\end{equation}
and
\begin{equation}
 h_1^{\prime\prime}(y_c)=0,
\end{equation}
it is easy from \eqref{eq:4-19-1} to see that $h_1^{\prime\prime}(y)$ is a monotonic decreasing function on an interval $[y_c,y_2]$, then
\begin{equation}
h_1^{\prime\prime}(y) \leq h_1^{\prime\prime}(y_c)=0, \ \ \ {\text{for}} \ y\in [y_c,y_2],
\end{equation}
that is, $h_1^{\prime}(y)$ is a monotonic decreasing function on an interval $[y_c,y_2]$. Since $h_1^{\prime}(y_c)=0$, we have
$h_1^{\prime}(y)\leq0$ for $y\in [y_c,y_2]$, that is, $h_1(y)$ is a monotonic decreasing function on an interval $[y_c,y_2]$. Since $h_1(y_2)=0$, we conclude that
\begin{equation}
h_1(y)\geq h_1(y_2)=0,\ \ \ {\mbox{for}} \ y\in [y_c,y_2],
\end{equation}
it follows that
\begin{equation}
\frac{1}{{\sqrt {f(y)} }} - \frac{1}{{\sqrt {g_1(y)} }} \ge 0,
\end{equation}
\begin{equation}
\int_{{y_c}}^{{y_2}} {\left( {\frac{1}{{\sqrt {f(y)} }} - \frac{1}{{\sqrt {g_1(y)} }}} \right)} dy \ge 0.
\end{equation}

On the other hand,
\begin{equation}
\int_{{y_c}}^{{y_2}} {\frac{1}{{\sqrt {g_1(y)} }}dy}  = {A_0}\pi,
\end{equation}
we have
\begin{equation}
\int_{{y_c}}^{{y_2}} {\frac{1}{{\sqrt {{y^{\frac{{2n - 2}}{{2n - 1}}}} - {y^{\frac{{2n}}{{2n - 1}}}} - a} }}dy}
=\int_{{y_c}}^{{y_2}} {{\frac{1}{{\sqrt {f(y)} }}}}dy\geq \int_{{y_c}}^{{y_2}} \frac{1}{{\sqrt {g_1(y)} }} dy={A_0}\pi.
\end{equation}

\noindent {\bf Case 2: $y\in [y_1,y_c]$.}

 By a direct calculation, we have
\begin{equation}
h_1^{\prime}(y) = g_1^{\prime}(y) - f^{\prime}(y) = c\left( {\sqrt {\frac{{{y_c}}}{y}}  - 1} \right) - \frac{{2n - 2}}{{2n - 1}}{y^{ - \frac{1}{{2n - 1}}}} + \frac{{2n}}{{2n - 1}}{y^{\frac{1}{{2n - 1}}}},
\end{equation}
\begin{equation}\label{eq:4-19-2}
h_1''(y) =  - \frac{c}{2}\sqrt {\frac{{{y_c}}}{{{y^3}}}}  + \frac{{2n - 2}}{{2n - 1}}\frac{1}{{2n - 1}}{y^{ - \frac{{2n}}{{2n - 1}}}} + \frac{{2n}}{{2n - 1}}\frac{1}{{2n - 1}}{y^{ - \frac{{2n - 2}}{{2n - 1}}}},
\end{equation}
and
\begin{equation}\label{eq:4-19-3}
h_1''(y_c)=0,
\end{equation}
then we obtain from \eqref{eq:4-19-2} and \eqref{eq:4-19-3} that
\begin{equation}
h_1''(y_c)\leq0,
\end{equation}
\begin{equation}
h_1''(y)\leq0,\ \ \ {\mbox{for}} \ y\in [y_1,y_c].
\end{equation}

From the above equations and  $h_1'(y_c)=0$, $h_1(y_1)=0$, we can get
\begin{equation}
h_1(y)\geq0,\ \ \ {\mbox{for}} \ y\in [y_1,y_c].
\end{equation}
Hence we have

\begin{equation}
\aligned
 &\ \ \ \int_{{y_1}}^{{y_c}} {\frac{1}{{\sqrt {f(y)} }}} dy\ge \int_{{y_1}}^{{y_c}} {\frac{1}{{\sqrt {g_1(y)} }}} dy=  \left( {1 - \frac{2}{\pi } + \frac{2}{\pi }\sqrt {\frac{{{y_1}}}{{{y_c}}}} } \right){A_0}\pi.
\endaligned
\end{equation}

 From the above two cases, we conclude that
\begin{equation}
\int_{{y_1}}^{{y_2}} {\frac{1}{{\sqrt {f(y)} }}} dy  \geq \left( {2-\dfrac{2}{\pi} + \frac{2}{\pi }\sqrt {\frac{{{y_1}}}{{{y_c}}}} } \right){A_0}\pi>
\left( {2-\dfrac{2}{\pi} } \right)  {A_0}\pi.
 \end{equation}
This completes the proof of the theorem \ref{theorem 1}.
$$\eqno{\Box}$$

\noindent{\it Proof of the corollary \ref{theorem 2}}.
From the lemma \ref{lemma 1} and lemma \ref{lemma 2},  we have the area $|M^n|$ of a compact minimal rotational hypersurface $M^n$  in $S^{n+1}(1)$ with non-constant principal curvatures except $M^n(3,2)$ is greater than
\begin{equation}
3\times 2\pi \sigma_{n-1}\dfrac{\inf\limits_{a\in(0,a_0)}\int_{x_1}^{x_2}\frac{x^{n-\frac{3}{2}} \,}{\sqrt{x^{n-1}-x^n-a}} dx}{K(a)}\,
\end{equation}
 since the rotation number of $M^n(3,2)$ is $2$ and the rotation numbers of other hypersurfaces are greater than $2$.

From the proof of the theorem \ref{theorem 1}, we can get that
\begin{equation}
3\times 2\pi \sigma_{n-1}\dfrac{\int_{x_1}^{x_2}\frac{x^{n-\frac{3}{2}} \,}{\sqrt{x^{n-1}-x^n-a}} dx}{K(a)}\ > 3(1-\dfrac{1}{\pi})\biggl|S^{1}(\sqrt{\frac{1}{n}})\times S^{n-1}(\sqrt{\frac{n-1}{n}})\biggl|.
\end{equation}
This completes the proof of the corollary \ref{theorem 2}.
$$\eqno{\Box}$$

%
%

\section{Estimate of an upper bound}

\begin{theorem}\label{theorem 4}
The area of $M^n(3,2)$ satisfies
\begin{equation}
|M^n(3,2)|<3\biggl|S^{1}\biggl(\sqrt{\frac{1}{n}}\biggl)\times S^{n-1}\biggl(\sqrt{\frac{n-1}{n}}\biggl)\biggl|.
\end{equation}
\end{theorem}

\begin{proof}
\noindent
Since the rotation number $p$ of $M^n(3,2)$ is $2$, we know from the lemma \ref{lemma 2} that the area of $M^n(3,2)$ is
\begin{equation}\label{eq:4-26-4}
4\pi \sigma_{n-1}\dfrac{\int_{x_1}^{x_2}\frac{x^{n-\frac{3}{2}} \,}{\sqrt{x^{n-1}-x^n-a}} dx}{K(a)}=
4\pi \sigma_{n-1}\dfrac{\int_{x_1}^{x_2}\frac{x^{n-\frac{3}{2}} \,}{\sqrt{x^{n-1}-x^n-a}} dx}{\frac{4\pi}{3}}
\end{equation}
for some $a<a_0$ and the area of $S^{1}\biggl(\sqrt{\frac{1}{n}}\biggl)\times S^{n-1}\biggl(\sqrt{\frac{n-1}{n}}\biggl)$ is $2\pi \sigma_{n-1}\sqrt{a_0}$. Hence it is sufficient to prove the following inequality
\begin{equation}
\int_{x_1}^{x_2}\dfrac{x^{n-\frac{3}{2}} \,}{\sqrt{x^{n-1}-x^n-a}} dx<2\pi\sqrt{a_0},
\end{equation}
that is,
\begin{equation}
\int_{y_1}^{y_2}\dfrac{1}{\sqrt{f(y)}}dy<(2n-1)\pi\sqrt{a_0}=\dfrac{4\pi}{\sqrt{2}}A_0.
\end{equation}

First of all, we construct a function $g_2(y)$ as follows:
\begin{equation}
g_2(y) = \left\{ {\begin{array}{*{20}{c}}
{C{{({y_1} - {y_c})}^2} - C{{(y - {y_c})}^2},\quad y \in [{y_1},{y_c}]},\\
{B{{({y_2} - {y_c})}^2} - B{{(y - {y_c})}^2},\quad y \in ({y_c},1]},
\end{array}} \right.\end{equation}
where $B=\frac{1}{2n-1}$, $C=\frac{{2(n - 1)}}{{{{(2n - 1)}^2}}}{\left( {\frac{{n - 1}}{n}} \right)^{ - n}}$ and $y_c=(\frac{n-1}{n})^{\frac{1}{2}(2n-1)}$, then we claim
\begin{equation}
g_2(y) \le f(y), \ \ \ y\in [{y_1},{y_2}].
\end{equation}

 Let $h_2(y)=g_2(y)-f(y)$. We have to consider two cases.

\noindent {\bf Case 1: $y\in [y_c,y_2]$.}

  By a direct calculation, we obtain
\begin{equation}\label{eq:4-22-1}
h_2'(y)=g_2'(y)-f'(y) = - 2B(y - {y_c}) - \frac{{2n - 2}}{{2n - 1}}{y^{\frac{{ - 1}}{{2n - 1}}}} + \frac{{2n}}{{2n - 1}}{y^{\frac{1}{{2n - 1}}}},
\end{equation}
\begin{equation}\label{eq:4-22-2}
h_2''(y) =  - 2B + \frac{{2n - 2}}{{2n - 1}}\frac{1}{{2n - 1}}{y^{ - \frac{{2n}}{{2n - 1}}}} + \frac{{2n}}{{2n - 1}}\frac{1}{{2n - 1}}{y^{ - \frac{{2n - 2}}{{2n - 1}}}}
\end{equation}
and
\begin{equation}\label{eq:4-22-3}
h_2''(1)=0,
\end{equation}
we can see from \eqref{eq:4-22-2} that $h_2^{\prime\prime}(y)$ is a monotonic decreasing function on an interval $[y_c,1]$, then
\begin{equation}
h_2^{\prime\prime}(y) \geq h_2^{\prime\prime}(y_2)\geq h_2^{\prime\prime}(1)=0, \ \ \ {\text{for}} \ y\in [y_c,y_2],
\end{equation}
that is, $h_2^{\prime}(y)$ is a monotonic increasing function on an interval $[y_c,y_2]$. Since $h_2^{\prime}(y_c)=0$, we have
$h_2^{\prime}(y)\geq0$ for $ y\in [y_c,y_2]$, that is, $h_2(y)$ is a monotonic increasing function on an interval $[y_c,y_2]$. Since $h_2(y_2)=0$, we conclude that
\begin{equation}
h_2(y)\leq h_2(y_2)=0,\ \ \ {\mbox{for}} \ y\in [y_c,y_2],
\end{equation}
it follows that
\begin{equation}
\frac{1}{{\sqrt {f(y)} }} - \frac{1}{{\sqrt {g_2(y)} }} \leq 0,
\end{equation}
\begin{equation}
\int_{{y_c}}^{{y_2}} {\left( {\frac{1}{{\sqrt {f(y)} }} - \frac{1}{{\sqrt {g_2(y)} }}} \right)} dy \leq 0.
\end{equation}

On the other hand,
\begin{equation}
\int_{{y_c}}^{{y_2}} {\frac{1}{{\sqrt {g_2(y)} }}dy}  =\frac{{\sqrt {2n - 1} }}{2} \pi,
\end{equation}
we have
\begin{equation}\label{eq:4-22-7}
\int_{{y_c}}^{{y_2}} {\frac{1}{{\sqrt {{y^{\frac{{2n - 2}}{{2n - 1}}}} - {y^{\frac{{2n}}{{2n - 1}}}} - a} }}dy}
=\int_{{y_c}}^{{y_2}} {{\frac{1}{{\sqrt {f(y)} }}}}dy\leq\int_{{y_c}}^{{y_2}} \frac{1}{{\sqrt {g_2(y)} }} dy=\frac{{\sqrt {2n - 1} }}{2}\pi.
\end{equation}\\

\noindent {\bf Case 2: $y\in [y_1,y_c]$.}

 By a direct calculation, we have
\begin{equation}\label{eq:4-22-4}
h_2'(y) =  - 2C(y - {y_c}) - \frac{{2n - 2}}{{2n - 1}}{y^{\frac{{ - 1}}{{2n - 1}}}} + \frac{{2n}}{{2n - 1}}{y^{\frac{1}{{2n - 1}}}},
\end{equation}
\begin{equation}\label{eq:4-22-5}
h_2''(y) =  - 2C + \frac{{2n - 2}}{{2n - 1}}\frac{1}{{2n - 1}}{y^{ - \frac{{2n}}{{2n - 1}}}} + \frac{{2n}}{{2n - 1}}\frac{1}{{2n - 1}}{y^{ - \frac{{2n - 2}}{{2n - 1}}}}
\end{equation}
and
\begin{equation}\label{eq:4-22-6}
h_2''(y_c)=0,
\end{equation}
then we obtain from \eqref{eq:4-22-5} that $h_2^{\prime\prime}(y)$ is a monotonic decreasing function on an interval $[y_1,y_c]$, then
\begin{equation}
h_2^{\prime\prime}(y) \geq h_2^{\prime\prime}(y_c)=0, \ \ \ {\text{for}} \ y\in [y_1,y_c],
\end{equation}
that is, $h_2^{\prime}(y)$ is a monotonic increasing function on an interval $[y_1,y_c]$. Since $h_2^{\prime}(y_c)=0$, we have
$h_2^{\prime}(y)\leq0$ for $y\in [y_1,y_c]$, that is, $h_2(y)$ is a monotonic decreasing function on an interval $[y_1,y_c]$. Since $h_2(y_1)=0$, we conclude that
\begin{equation}
h_2(y)\leq h_2(y_1)=0,\ \ \ {\mbox{for}} \ y\in [y_1,y_c],
\end{equation}
it follows that
\begin{equation}
\frac{1}{{\sqrt {f(y)} }} - \frac{1}{{\sqrt {g_2(y)} }} \leq 0.
\end{equation}

 On the other hand,
\begin{equation}
\int_{{y_1}}^{{y_c}} {\frac{1}{{\sqrt {g_2(y)} }}dy} =\frac{1}{2}\sqrt {\frac{{{{(2n - 1)}^2}}}{{2(n - 1)}}{{\left( {\frac{n}{{n - 1}}} \right)}^{ - n}}}\pi,
\end{equation}
we have

\begin{equation}\label{eq:4-22-8}
\aligned
&\ \ \ \int_{{y_1}}^{{y_c}} {\frac{1}{{\sqrt {{y^{\frac{{2n - 2}}{{2n - 1}}}} - {y^{\frac{{2n}}{{2n - 1}}}} - a} }}dy}\\
&=\int_{{y_1}}^{{y_c}} {{\frac{1}{{\sqrt {f(y)} }}}}dy\\
 &\leq\int_{{y_1}}^{{y_c}} \frac{1}{{\sqrt {g_2(y)} }} dy\\
&= \frac{1}{2}\sqrt {\frac{{{{(2n - 1)}^2}}}{{2(n - 1)}}{{\left( {\frac{n}{{n - 1}}} \right)}^{ - n}}}\pi.
\endaligned
\end{equation}

\noindent From \eqref{eq:4-22-7} and \eqref{eq:4-22-8}, we have
\begin{equation}\label{eq:4-22-9}
\aligned
 \int_{{y_1}}^{{y_2}} {\frac{1}{{\sqrt {f(y)} }}} dy
&\le \frac{{\sqrt {2n - 1} }}{2}\pi  + \frac{1}{2}\sqrt {\frac{{{{(2n - 1)}^2}}}{{2(n - 1)}}{{\left( {\frac{n}{{n - 1}}} \right)}^{ - n}}} \pi\\
 & = \left( {1 + \sqrt {\frac{{2(n - 1)}}{{2n - 1}}{{\left( {\frac{n}{{n - 1}}} \right)}^n}} } \right){A_0}\pi\\
 &= \left( 1 +\sqrt{\frac{2}{(2n-1)a_0}}\right){A_0}\pi\\
 &<\left( 1 +\sqrt{\biggl(\frac{n}{n-1}\biggl)^n}\right){A_0}\pi.
\endaligned
\end{equation}

When $n=3$, we see from \eqref{eq:4-22-9}
\begin{equation}\label{eq:4-26-1}
\int_{{y_1}}^{{y_2}} {\frac{1}{{\sqrt {f(y)} }}} dy\leq\left( {1 + \sqrt {\frac{{2(n - 1)}}{{2n - 1}}{{\left( {\frac{n}{{n - 1}}} \right)}^n}} } \right){A_0}\pi
 =\left(1+\sqrt{\frac{27}{10}}\right){A_0}\pi,
 \end{equation}
\noindent when $n\geq4$, we get
\begin{equation}\label{eq:4-26-2}
 \int_{{y_1}}^{{y_2}} {\frac{1}{{\sqrt {f(y)} }}} dy<\left( 1 +\sqrt{\biggl(\frac{n}{n-1}\biggl)^n}\right){A_0}\pi\leq (1+\dfrac{16}{9}){A_0}\pi=\frac{25}{9}{A_0}\pi.
 \end{equation}

Hence, we obtain from \eqref{eq:4-26-1} and \eqref{eq:4-26-2} that
\begin{equation}\label{eq:4-26-3}
\int_{{y_1}}^{{y_2}} {\frac{1}{{\sqrt {f(y)} }}} dy<\frac{25}{9} A_0\pi<\frac{4}{\sqrt{2}}A_0\pi.
\end{equation}
\noindent This completes the proof of the theorem \ref{theorem 4}.
\end{proof}

\noindent{\it Proof of the theorem \ref{theorem 3}}.
From the lemma \ref{lemma 1} and the lemma \ref{lemma 2}, we know that the area  of a compact minimal rotational hypersurface $M^n$ in $S^{n+1}(1)$ with non-constant principal curvatures except $M^n(3,2)$ and $M^n(5,3)$  is greater than
\begin{equation}
4\times 2\pi\sigma_{n-1}\dfrac{\inf\limits_{a\in (0,a_0)}\int_{x_1}^{x_2}\frac{x^{n-\frac{3}{2}} \,}{\sqrt{x^{n-1}-x^n-a}} dx}{K(a)}\,
\end{equation}
since the rotation number of $M^n(3,2)$ is $2$, the rotation number of $M^n(5,3)$ is $3$, the rotation number of $M^n(7,4)$ is $4$ and the rotation numbers of other hypersurfaces are greater than $4$.

We know from the lemma \ref{lemma 2} and the proof of the theorem \ref{theorem 1} that

\begin{equation}
\aligned
|M^n(7,4)|&=4\times 2\pi\sigma_{n-1}\dfrac{\int_{x_1}^{x_2}\frac{x^{n-\frac{3}{2}} \,}{\sqrt{x^{n-1}-x^n-a}} dx}{\frac{8}{7}\pi}\ \\
&>4(1-\frac{1}{\pi})\times\frac{\sqrt{2}}{\frac{8}{7}}\biggl|S^{1}(\sqrt{\frac{1}{n}})\times S^{n-1}(\sqrt{\frac{n-1}{n}})\biggl|\\
&>3|S^{1}(\sqrt{\frac{1}{n}})\times S^{n-1}(\sqrt{\frac{n-1}{n}})\biggl|\\
&>|M^n(3,2)|
\endaligned
\end{equation}
for some $a\in (0,a_0)$ and the area of other hypersurface $M^n$ except $M^n(3,2)$, $M^n(5,3)$ and $M^n(7,4)$ satisfies
\begin{equation}
\aligned
|M^n|&>5(1-\frac{1}{\pi})\biggl|S^{1}(\sqrt{\frac{1}{n}})\times S^{n-1}(\sqrt{\frac{n-1}{n}})\biggl|\\
&>3|S^{1}(\sqrt{\frac{1}{n}})\times S^{n-1}(\sqrt{\frac{n-1}{n}})\biggl|\\
&>|M^n(3,2)|.
\endaligned
\end{equation}
Hence, the lowest value of area among all compact minimal rotational hypersurfaces with non-constant principal curvatures in the unit sphere $S^{n+1}(1)$ is the area of either $M^n(3,2)$ or $M^n(5,3)$.
\noindent This completes the proof of the theorem \ref{theorem 3}.
$$\eqno{\Box}$$

 According to the above theorems, we propose the following conjecture.

\noindent {\bf Conjecture}: {\it The lowest value of area among all compact minimal rotational hypersurfaces with non-constant principal curvatures in the unit sphere $S^{n+1}(1)$ is the area of $M^n$ with $3$-fold rotational symmetry and rotation number $2$.}

\section{Entropies of some special self-shrinkers}

In this section, we estimate the entropies of some special self-shrinkers as the application of the estimate of the areas.

 An immersed hypersurface $X:M^n\rightarrow \mathbb{R}^{n+1}$ in the ($n+1$)-dimensional Euclidean space $\mathbb{R}^{n+1}$ is called a {\it self-shrinker} if it satisfies
\begin{equation}\label{eq:6-3-1}
H+\langle X,N\rangle=0,
\end{equation}
where $N$ is the unit normal vector of $X:M^n\rightarrow \mathbb{R}^{n+1}$, $H$ is the mean curvature.

From the definition of self-shrinkers, we know that if $X:M^n\rightarrow S^{n+1}(1)$ is a minimal rotational hypersurface, then $C(M^n)$, the {\it cone} over $M^n$, satisfies the self-shrinker equation \eqref{eq:6-3-1} in $\mathbb{R}^{n+2}$.

On the other hand, the {\it entropy} $\lambda(M^n)$ of self-shrinker $M^n$ can be defined as follows:
\begin{equation}
\lambda(M^n)=\dfrac{1}{(2\pi)^{n/2}}\int_{M^n}e^{-|X|^2/2}d\mu.
\end{equation}

Therefore, the entropy $\lambda(C(M^n))$ of the cone $C(M^n)$ over a compact minimal rotational hypersurface $M^n\subset S^{n+1}(1)$ in $\mathbb{R}^{n+2}$ is

\begin{equation}\label{eq:6-3-2}
\aligned
 \lambda(C(M^n))&=\frac{1}{(2\pi)^{(n+1)/2}}|M^n|\int_0^{+\infty}t^ne^{-t^2/2}dt\\
 &=\dfrac{1}{2}\pi^{-(n+1)/2}\Gamma(\frac{n+1}{2})|M^n|\\
&=\dfrac{1}{\sigma_n}|M^n|,
\endaligned
\end{equation}
where $|M^n|$ denotes the area of $M^n$, $\Gamma(x)=\int_0^{+\infty}t^{x-1}e^{-t}dt$ is the Gamma function, $\sigma_n$ denotes the $n$-area of $S^n(1)$.

By a computation, we have

\begin{theorem}\label{theorem 5}
If $M^n$ is a compact minimal rotational hypersurface in $S^{n+1}(1)$, then the entropies $\lambda(C(M^n))$ of the cones $C(M^n)$ over $M^n$ in $\mathbb{R}^{n+2}$ satisfies either
$\lambda(C(M^n))=1$, or $\lambda(C(M^n))=\dfrac{2\pi\sigma_{n-1}\sqrt{a_0}}{\sigma_n}$,  or
$\lambda(C(M^n))> \dfrac{4(\pi-1)\sigma_{n-1}\sqrt{a_0}}{\sigma_n}$, where $a_0=\frac{(n-1)^{n-1}}{n^n}$, $\sigma_n$ denotes the $n$-area of $S^n(1)$.
\end{theorem}

\begin{proof}
Combining the theorem \ref{theorem 1}, \eqref{eq:6-3-2} and using
\begin{equation}
\biggl|S^{1}(\sqrt{\frac{1}{n}})\times S^{n-1}(\sqrt{\frac{n-1}{n}})\biggl|=2\pi\sigma_{n-1}\sqrt{a_0},
\end{equation}
we can proof the theorem \ref{theorem 5}.
\end{proof}

\end{document}